\documentclass{amsart}
\usepackage{enumerate}
\usepackage{amssymb}
\usepackage[headings]{fullpage}
\DeclareSymbolFont{EulerScriptBold}{U}{eus}{b}{n}
\DeclareSymbolFontAlphabet\eusb{EulerScriptBold}
\def\BB{\mathbb B}
\def\bu{$\bullet$\quad}
\def\CC{\mathbb C}
\def\CF{\mathcal F}
\def\CG{\mathcal G}
\def\CH{\mathcal H}
\def\CO{\mathcal O}
\def\const{\operatorname{const}}
\def\DD{\mathbb D}
\def\ES{\eusb S}
\def\FR{\mathfrak R}
\def\halfskip{\vskip 6pt plus 1pt minus 1pt}
\def\id{\operatorname{id}}
\def\LL{\mathbb L}
\def\OO{\mathbb O}
\def\pder#1#2{\frac{\partial#1}{\partial#2}}
\def\phi{\varphi}
\def\Phi{\varPhi}
\def\rho{\varrho}
\def\RR{\mathbb R}
\def\SS{\mathbb S}
\def\sprod#1#2{\langle #1,#2 \rangle}
\def\str{\operatorname{Str}}
\def\too{\longrightarrow}
\def\tuu{\longmapsto}
\def\wdht{\widehat}
\def\wdtl{\widetilde}
\makeatletter
\def\th@mytheorem{%
  \let\thm@indent\noindent
  \thm@headfont{\bfseries}
    \itshape
}
\def\th@myremark{%
  \let\thm@indent\noindent
  \thm@headfont{\bfseries}
}

\def\@makefnmark{\hbox{$\left(^{\@thefnmark}\right)$\;}}
\makeatother
\theoremstyle{mytheorem}
\newtheorem{Theorem}{Theorem}[section]
\newtheorem{Lemma}[Theorem]{Lemma}
\newtheorem{Proposition}[Theorem]{Proposition}
\theoremstyle{myremark}
\newtheorem{Example}[Theorem]{Example}
\newenvironment{myenumerate}{
\begin{enumerate}[{\rm (a)}]
  \setlength{\itemsep}{1pt}
  \setlength{\parskip}{0pt}
  \setlength{\parsep}{0pt}
}{\end{enumerate}}

\newenvironment{myrenumerate}{
\begin{enumerate}[{\rm (i)}]
  \setlength{\itemsep}{1pt}
  \setlength{\parskip}{0pt}
  \setlength{\parsep}{0pt}
}{\end{enumerate}}

\begin{document}
\title{A note on envelopes of holomorphy}

\author[M.~Jarnicki]{Marek Jarnicki}
\address{Jagiellonian University, Faculty of Mathematics and Computer Science, Institute of Mathematics,
{\L}ojasiewicza 6, 30-348 Krak\'ow, Poland} 
\email{Marek.Jarnicki@im.uj.edu.pl}

\author[P.~Pflug]{Peter Pflug}
\address{Carl von Ossietzky Universit\"at Oldenburg, Institut f\"ur Mathematik,
Postfach 2503, D-26111 Oldenburg, Germany}
\email{Peter.Pflug@uni-oldenburg.de}
\thanks{The research was partially supported by
grant no.~UMO-2011/03/B/ST1/04758 of the Polish National Science Center (NCN)}

\begin{abstract}
Let $p:X\too M$ be a Riemann domain over a connected $n$-dimensional complex submanifold $M$ of $\CC^N$ and let $\CF\subset\CO(X)$ be such that $p\in\CF^N$.
Our aim is to discuss relations between the $\CF$-envelope of holomorphy of $(X,p)$ 
in the sense of Riemann domains over $M$ and the $\CF$-envelope of holomorphy of $X$ in the sense of complex manifolds. 
\end{abstract}

\subjclass[2010]{32D10, 32D15, 32D25}

\keywords{Riemann domain, envelope of holomorphy}

\maketitle

\section{Introduction}
Let $M$ be a connected $n$-dimensional complex submanifold of $\CC^N$, let $(X,p)$ ($p:X\too M$) be a Riemann domain over $M$, and let $\CF\subset\CO(X)$ be such that $p\in\CF^N$ (cf.~\S\;\ref{sectionRiemannDom}).
Our aim is to discuss the notion of the $\CF$-envelope of holomorphy of $X$. More precisely, we like to discuss relations between the $\CF$-envelope of holomorphy of $(X,p)$ 
in the sense of Riemann domains over $M$ and the $\CF$-envelope of holomorphy of $X$ in the sense of complex manifolds. We will see that, even in the case of domains in $\CC^n$, both approaches
lead to some fundamental difficulties. Notice that the problem does not appear in the case where $\CF=\CO(X)$.

Let $\phi:(X,\CF)\too(\wdtl X,\wdtl\CF)$ be the $\CF$-envelope of holomorphy of $X$ in the sense of complex manifolds (cf.~\cite{Vig1982}, see also \S\;\ref{sectionEnvOverComMan}) 
and let $\wdtl p\in\wdtl{\CF}^N$ be such that $\wdtl p\circ\wdtl\phi\equiv p$. 
Observe that $\wdtl p:\wdtl X\too M$ (cf.~the proof of Theorem \ref{ThmThThmStein}). Put $Z_p:=\{a\in\wdtl X: \wdtl p$ is not biholomorphic near $a\}$; 
the set $Z_p$ is an analytic subset of $\wdtl X$ with $\dim Z_p\leq n-1$. The following result characterizes
relations between the $\CF$-envelopes of holomorphy of $X$.

\begin{Theorem}[cf.~Theorem \ref{ThmMain}]
Under the above notation, $\phi:(X,p)\too(\wdtl X\setminus Z_p,\wdtl p)$ is the $\CF$-envelope of holomorphy in the sense of Riemann domains over $M$. Moreover, if $\CF=\CO(X)$, then $Z_p=\varnothing$.
In particular, if $\CF=\CO(X)$, then $(\wdtl X,\wdtl p)$ must be a Stein Riemann domain over $M$.
\end{Theorem}
The case where $\CF=\CO(X)$ has been discussed in \cite{Ker1959}. The proof will be given in \S\;\ref{sectionMainResult}.

Using the above result we get the following examples (details will be given in \S\S\;\ref{sectionExmpl}, \ref{sectionExmpl-2} ).

\begin{Example}\label{Exmpl} Let $\DD\subset\CC$ be the unit disc and let $\DD_\ast:=\DD\setminus\{0\}$.
Take $M:=\CC$, $X:=\DD_\ast$, $\CF:=\CH^\infty(\DD_\ast)$. Then $\id:(\DD_\ast,\CH^\infty(\DD_\ast))\too(\DD,\CH^\infty(\DD))$ is the $\CH^\infty(\DD_\ast)$-envelope of holomorphy of $\DD_\ast$ in the sense of complex manifolds.

\bu If $p:=\id$, then $Z_p=\varnothing$. 

\bu If $p:=\id^2$, then $Z_p=\{0\}$.

\noindent Consequently, the $\CH^\infty(\DD_\ast)$-envelope of holomorphy of $(\DD_\ast,p)$  depends on a particular choice of the projection $p$ (notice that $\DD$ and $\DD_\ast$ are not homeomorphic). 
\end{Example}

\begin{Example}\label{Exmpl-2}
Let $M:=\CC$, $X:=\CC\setminus((-\infty,1]\cup[1,+\infty))$ (note that $X$ is simply connected), $\CF=\{f_1,f_2\}:=\{\id,X\ni\lambda\tuu\sqrt{1-\lambda^2}\}$ (the branch of the square root is 
arbitrarily fixed), 
$$
\phi:=(f_1,f_2):X\too\wdtl X:=\{(z_1,z_2)\in\CC^2: z_1^2+z_2^2=1\}
$$ 
($\wdtl X$ is a connected one dimensional complex manifold),
$\wdtl\CF=\{\wdtl f_1,\wdtl f_2\}:=\{z_1|_{\wdtl X},z_2|_{\wdtl X}\}$. Then 
$$
\phi:(X,\CF)\too(\wdtl X,\wdtl\CF)
$$ 
is the $\CF$-envelope of holomorphy in the sense of complex manifolds.

Taking $p:=f_1=\id$, we get $\wdtl p=\wdtl f_1$, which implies that $Z_p=\{(0,-1), (0,+1)\}$. Consequently, 
$$
\phi:(X,\id)\too(\wdtl X\setminus\{(0,-1), (0,+1)\},z_1)
$$ 
is the $\CF$-envelope of holomorphy in the sense of Riemann domains.
\end{Example}

The above examples might suggest that the $\CF$-envelope of holomorphy in the sense of complex manifold is perhaps better.
Unfortunately, the following result shows that this is not so.

\begin{Theorem}\label{ThmIsNotStein}
For every $n\geq2$ there exists a domain $X\subset\CC^n$ and a family $\CF\subset\CO(X)$ such that the $\CF$-envelope of holomorphy of $X$ in the sense of complex manifolds is neither Stein nor a Riemann
domain over $\CC^n$.
\end{Theorem}
The proof will be given in \S\;\ref{sectionThmIsNotStein}.

\section{Riemann domains over complex manifolds}\label{sectionRiemannDom}
The aim of this section is to recall basic terminology related to Riemann domains over complex manifolds --- cf.~e.g.~\cite{JarPfl2000}.
Let $M$ be a connected $n$-dimensional complex manifold (e.g.~$M=\CC^n$). Denote by $\FR(M)$ the family of all \emph{Riemann regions over $M$}, i.e.~the family of all pairs $(X,p)$, where $X$ is a Hausdorff
topological space and $p:X\too M$ is locally homeomorphic (each point $a\in X$ has an open neighborhood $U$ such that $p(U)$ is open in $M$ and $p|_U:U\too p(U)$ is homeomorphic).
The \emph{projection} $p$ introduces on $X$ a structure $\str(X,p)$ of an $n$-dimensional complex manifold (such that $p$ is locally biholomorphic).
Let $\FR_c(M)$ be the subfamily of those $(X,p)\in\FR(M)$ for which $X$ is connected. 

Let $(X,p),(Y,q)\in\FR(M)$. We say that a continuous mapping $\phi:X\too Y$ is a \emph{morphism} if $q\circ\phi\equiv p$. We shortly write ``$\phi:(X,p)\too(Y,q)$ is a  morphism''.
Each  morphism is locally biholomorphic. We say that a morphism $\phi:X\too Y$ is an \emph{isomorphism} if $\phi$ is bijective and $\phi^{-1}:(Y,q)\too (X,p)$ is also a  morphism.
Notice that a morphism $\phi:(X,p)\too(Y,q)$ is an isomorphism iff $\phi$ is bijective. Each locally biholomorphic mapping $\phi:X\too Y$ induces a homomorphism 
$\CO(Y)\ni g\overset{\phi^\ast}\tuu g\circ\phi\in\CO(X)$. Observe that $\phi^\ast$ is injective iff each connected component of $Y$ intersects $\phi(X)$.  If $\phi^\ast$ is injective and 
$g\circ\phi=f$, then we write $g=f^\phi$.

Let $(X,p), (Y,q)\in\FR_c(M)$ and $\varnothing\neq\CF\subset\CO(X)$. We say that a  morphism $\phi:(X,p)\too(Y,q)$ is an \emph{$\CF$-extension} if $\CF\subset\phi^\ast(\CO(Y))$. We put
$\CF^\phi:=\{f^\phi: f\in\CF\}$. Note that $\phi^\ast|_{\CF^\phi}:\CF^\phi\too\CF$ is bijective. We say that an $\CF$-extension $\phi:(X,p)\too(\wdht X,\wdht p)$ is an 
\emph{$\CF$-envelope of holomorphy} if for every $\CF$-extension $\psi:(X,p)\too(Y,q)$ there exists a  morphism $\sigma:(Y,q)\too(\wdht X,\wdht p)$ such that $\sigma\circ\psi\equiv\phi$.
Observe that in fact $\sigma:(Y,q)\too(\wdht X,\wdht p)$ is  an $\CF^\psi$-extension. 
Such an $\CF$-envelope of holomorphy is uniquely determined up to an isomorphism. By the Thullen theorem the $\CF$-envelope of holomorphy always exists ---
cf.~\cite{JarPfl2000}, Theorem 1.8.4 for the case $M=\CC^n$ (the general case goes in the same way).  
Moreover, if $M$ is Stein, then the $\CF$-envelope of holomorphy is also Stein --- cf.~\cite{JarPfl2000}, Cartan--Thullen Theorem 1.10.4 (the case $M=\CC^n$) and Theorem \ref{ThmThThmStein} 
(the general case). We say that $(X,p)$ is an \emph{$\CF$-domain of holomorphy} if $\id:(X,p)\too(X,p)$ is the $\CF$-envelope of holomorphy.

Our main problem is to discuss the following situation.
Suppose that $(X,p),  (X,q)\in\FR_c(M)$ are such that $\str(X,p)=\str(X,q)$ (equivalently: $q$ is holomorphic in the sense of $\str(X,p)$ and $p$ is holomorphic in $\str(X,q)$). 
Let $\phi_p:(X,p)\too(\wdht X_p,\wdht p)$, $\phi_q:(X,q)\too(\wdht X_q, \wdht q)$ be the $\CF$-envelopes of holomorphy.
We are interested in the situation when there exists a biholomorphic mapping  $\tau:\wdht X_p\too\wdht X_q$ such that $\tau\circ\phi_p\equiv\phi_q$.
It is known that in the case where $M=\CC^n$, $\CF=\CO(X)$ such a biholomorphic mapping  $\tau$ exists (cf.~\cite{JarPfl2000}, 
Theorem 2.12.1). Consequently, if $M=\CC^n$ and $\CF=\CO(X)$, then the $\CO(X)$-envelope of holomorphy depends only on the complex structure $\str(X,p)$.

The next problem appears when $p:X\too M$ and $q:X\too M'$ are Riemann domains over different connected $n$-dimensional Stein manifolds such that $\str(X,p)=\str(X,q)$. 
Observe that $\id:(M,\id)\too(M,\id)$ is the $\CF$-envelope of holomorphy over $M$ for \emph{any} $\CF$. This may lead to some pathological situations, 
e.g.~$\id:(\DD_\ast,\id)\too(\DD_\ast,\id)$ is the $\CH^\infty(\DD_\ast)$-envelope of holomorphy over $M=\DD_\ast$ but not over $M=\CC$.
If $M$ and $M'$ are biholomorphic, then the situation is simple: if $\Phi:M\too M'$ is biholomorphic, then the mapping $\FR(M)\ni(X,p)\tuu(X,\Phi\circ p)\in\FR(M')$ is bijective and
$\str(X,p)=\str(X,\Phi\circ p)$; moreover, $\phi:(X,p)\too(Y,q)$ is an $\CF$-extension (resp.~$\CF$-envelope of holomorphy) over $M$ iff
$\phi:(X,\Phi\circ p)\too(Y,\Phi\circ q)$ is an $\CF$-extension (resp.~$\CF$-envelope of holomorphy) over $M'$. In particular, by the Remmert embedding theorem, we may always assume that 
our Stein manifold $M$ is a connected $n$-dimensional complex submanifold of $\CC^N$.

\section{Domains over Stein manifolds vs.~domains over $\CC^N$}
The aim of this section (inspired by \cite{Ros1963}) is to show that in fact many properties of the Riemann domains over Stein manifolds may be easily deduced from the corresponding properties of 
Riemann domains over $\CC^N$.
Let $M$ be a connected $n$--dimensional complex submanifold of $\CC^N$ and
let $\sigma:S\too M$ be a holomorphic retraction, where $S\subset\CC^N$ is a domain (cf.~\cite{GunRos1965}, Ch.~VIII,~C, Theorem 8). 

\begin{Lemma}\label{LemX0}
Let $(X,p)\in\FR(M)$. Put 
\begin{gather*}
\wdtl X_0:=\{(x,v)\in X\times\CC^N: p(x)+v\in S,\;\sigma(p(x)+v)=p(x)\},\quad p_0:\wdtl X_0\too S,\; p_0(x,v):=p(x)+v.
\end{gather*}
Then $(\wdtl X_0,p_0)\in\FR(\CC^N)$.
\end{Lemma}

Observe that $X\times\{0\}\subset\wdtl X_0$ and $p_0(x,0)=p(x)$.

\begin{proof}
Fix a point $(x_0,v_0)\in\wdtl X_0$. Let $U\subset X$ be an open neighborhood of $x_0$ such that $p|_U:U\too p(U)$ is biholomorphic.
Define $V:=\wdtl X_0\cap(U\times\CC^N)$, $V':=\sigma^{-1}(p(U))$. It suffices to show that
$p_0|_V:V\too V'$ is homeomorphic. Clearly, $p_0(V)\subset V'$. Define $\rho:V'\too X\times\CC^N$,
$\rho(z):=((p|_U)^{-1}(\sigma(z)), z-\sigma(z))$. Observe that $\rho(V')\subset V$. Moreover,
$p_0\circ\rho=\id_{V'}$ and $\rho\circ p_0=\id_V$.
\end{proof}

As a direct corollary we get the following proposition.

\begin{Proposition}\label{PropX0}
Let $(X,p)\in\FR_c(M)$ and let $(\wdtl X_0,p_0)$ be as in Lemma \ref{LemX0}. Let $X_0$ be the connected component of $\wdtl X_0$ that contains $X\times\{0\}$. Then:

\bu $(X_0,p_0)\in\FR_c(\CC^N)$,

\bu $X\times\{0\}\simeq X$ is a submanifold of $X_0$,

\bu the mapping $X_0\ni(x,v)\overset{\sigma_X}\tuu x\in X\simeq X\times\{0\}$ is a holomorphic retraction.
\end{Proposition}

\begin{Lemma}\label{LemXY}
Let $(X,p), (Y,q)\in\FR_c(M)$, $\varnothing\neq\CF\subset\CO(X)$, and let $\phi:(X,p)\too(Y,q)$ be an $\CF$-extension (over $M$).
Assume that $(X_0,p_0)$ and $(Y_0,q_0)$ are constructed according to Proposition \ref{PropX0}.
Put $X_0\ni(x,v)\overset{\phi_0}\tuu(\phi(x),v)\in Y\times\CC^N$, $\CF_0:=\{f\circ\sigma_X: f\in\CF\}\subset\CO(X_0)$.
Then $\phi_0:(X_0,p_0)\too(Y_0,q_0)$ is an $\CF_0$-extension (over $\CC^N$).
\end{Lemma}

\begin{proof}
 First observe that $\phi_0: X_0\too\wdtl Y_0$ is well defined: $q(\phi(x))+v=p(x)+v\in S$ and
$\sigma(q(\phi(x))+v)=\sigma(p(x)+v)=p(x)=q(\phi(x))$. It is clear that $q_0\circ\phi_0=p_0$ and $\phi_0(X\times\{0\})\subset Y\times\{0\}$.
Thus $\phi_0(X_0)\subset Y_0$. Moreover, for each $f\in\CF$ we have $(f^\phi\circ\sigma_Y)\circ\phi_0=f^\phi\circ\phi=f=f\circ\sigma_X$.
\end{proof}

Let $\psi:(X_0,p_0)\too(Z,r)$ be an $\CF_0$-extension (over $\CC^N$). Put $Z^M:=r^{-1}(M)$. Observe that $(Z^M,r)\in\FR(M)$ and $\psi(X\times\{0\})\subset Z^M$. 
Let $Z^{M,\psi}$ be the connected component of $Z^M$ that contains $\psi(X\times\{0\})$.  
Then $\psi:(X,p)\too(Z^{M,\psi},r)$ is an $\CF$-extension (over $M$); recall that $X\simeq X\times\{0\}$.

\begin{Theorem}\label{ThmThThmStein}
Under the above notation if $\psi:(X_0,p_0)\too(Z,r)$ is the $\CF_0$-envelope of holomorphy (over $\CC^N$), 
then  $\psi:(X,p)\too(Z^{M,\psi},r)$ is the $\CF$-envelope of holomorphy (over $M$). Moreover, $Z^{M,\psi}$ is Stein.

Consequently, for any Riemann domain $(X,p)\in\FR_c(M)$ and for any family of functions $\varnothing\neq\CF\subset\CO(X)$, if $\phi:(X,p)\too(\wdht X,\wdht p)$ is the $\CF$-envelope of holomorphy 
(over $M$), then $\wdht X$ is Stein.
\end{Theorem}

\begin{proof}
We only need to show that the extension $\psi:(X,p)\too(Z^{M,\psi},r)$
is maximal. Suppose that $\phi:(X,p)\too(Y,q)$ is an $\CF$-extension (over $M$).
Then, by Lemma \ref{LemXY}, $\phi_0:(X_0,p_0)\too(Y_0,q_0)$ is an $\CF_0$-extension (over $\CC^N$).
Thus, there exists a morphism $\tau: (Y_0,q_0)\too(Z,r)$ such that $\tau\circ\phi_0=\psi$.
It remains to observe that $\tau(Y\times\{0\})\subset Z^{M,\psi}$.

We know that $Z$ is a Stein manifold (cf.~\cite{JarPfl2000}, Cartan--Thullen Theorem 1.10.4). 
Let $M=\{\zeta\in\CC^N: g_j(\zeta)=0,\;j=1,\dots,k\}$, where $g_1,\dots,g_k\in\CO(\CC^N)$. Then $Z^M=\{z\in Z: g_j\circ r(z)=0,\;j=1,\dots,k\}$.
Hence $Z^M$ is a submanifold of $Z$ and, therefore, $Z^M$ is Stein. Consequently, $Z^{M,\psi}$ is Stein.
\end{proof}

\section{$\CF$-envelopes in the sense of complex manifolds}\label{sectionEnvOverComMan} 
The aim of this section is to recall a more general notion of the $\CF$-envelope of holomorphy (cf.~\cite{Ker1959}, \cite{Vig1982}).
Let $\ES$ be the family of all pairs $(Y,\CG)$ such that:

\bu $Y$ is a connected $n$-dimensional complex manifold,

\bu $\CG\subset\CO(Y)$, 

\bu for every $a\in Y$ there exist a $g\in\CG^n$ and an open neighborhood $U$ of $a$ such that
$g(U)$ is open and $g|_U:U\too g(U)$ is biholomorphic.

One may prove that if $(Y,\CG)$ in $\ES$, then $Y$ is countable at infinity (cf.~\cite{Gra1955}).
Observe that if $M$ is a connected submanifold of $\CC^N$, $(Y,q)\in\FR_c(M)$, and $\CG\subset\CO(Y)$ is such that $q\in\CG^N$, then $(Y,\CG)\in\ES$.
We fix a pair $(X,\CF)\in\ES$.
Let $(Y,\CG)$, $(Z,\CH)\in\ES$. We say that a holomorphic mapping
$\phi:Y\too Z$ is a \emph{$C$-morphism} if $\phi^\ast|_{\CH}:\CH\too\CG$ is bijective. Observe that $\phi$ must be locally biholomorphic. 
We write ``$\phi:(Y,\CG)\too(Z,\CH)$ is a $C$-morphism''. Note that if $\phi:(X,p)\too(Y,q)$ is an $\CF$-extension in the sense of Riemann domains over $M$ with $p\in\CF^N$, then
$\phi:(X,\CF)\too(Y,\CF^\phi)$ is a $C$-morphism.

We say that a $C$-morphism $\phi:(X,\CF)\too(\wdtl X,\wdtl{\CF})$ is an \emph{$\CF_C$-envelope of holomorphy} if for every $C$-morphism $\psi:(X,\CF)\too(Y,\CG)$ there exists a holomorphic
mapping $\sigma:Y\too\wdtl X$ such that $\sigma\circ\psi\equiv\phi$. Notice that in fact $\sigma:(Y,\CG)\too(\wdtl X,\wdtl{\CF})$ is a $C$-morphism. 
Such an $\CF_C$-envelope of holomorphy is uniquely determined up to a $C$-isomorphism. 
It is clear that the $\CF_C$-envelope has no defects of the $\CF$-envelope in the sense of Riemann domains, i.e.~it depends only on $\CF$.

\begin{Theorem}[cf.~\cite{Vig1982}]
For arbitrary $(X,\CF)\in\ES$ the $\CF_C$-envelope of holomorphy exists. 
\end{Theorem}

\section{Main result}\label{sectionMainResult}

Let $(X,p)\in\FR_c(M)$, where $M$ is a connected submanifold of $\CC^N$. Let $\CF\subset\CO(X)$ be such that $p\in\CF^N$. 
Assume that $\phi:(X,\CF)\too(\wdtl X,\wdtl{\CF})$ is the $\CF_C$-envelope of holomorphy. 

Let $\wdtl p\in\wdtl{\CF}^N$ be such that $\wdtl p\circ\phi\equiv p$; note that $\wdtl p:\wdtl X\too M$.
Put $Z_p:=\{a\in\wdtl X: \wdtl p$ is not biholomorphic near $a\}$. Notice that $Z_p$ is an analytic subset of $\wdtl X$ with $\dim Z_p\leq n-1$.

\begin{Theorem}\label{ThmMain}
Under the above notation we have:
\begin{myenumerate}
\item $\phi: (X,p)\too (\wdtl X\setminus Z_p,\wdtl p)$ is the $\CF$-envelope of holomorphy (over $M$).

\item $Z_p=\varnothing$ iff the $\CF$- and $\CF_C$-envelope of holomorphy coincide. 

\item If $\CF=\CO(X)$, then the $\CO(X)$- and $\CO(X)_C$-envelope of holomorphy coincide.
\end{myenumerate}
\end{Theorem}

\begin{proof}
(a)
Let $\wdht\phi:(X,p)\too(\wdht X,\wdht p)$ be the $\CF$-envelope of holomorphy in the sense of Riemann domains over $M$. 
Then $\wdht\phi:(X,\CF)\too(\wdht X,\CF^\phi)$ is a $C$-morphism.
Consequently, there exists a $C$-morphism $\sigma:(\wdht X,\CF^\phi)\too(\wdtl X,\wdtl{\CF})$ such that $\sigma\circ\wdht\phi\equiv\phi$. 

On the other hand, $(\wdtl X\setminus Z_p,\wdtl p)$ is a Riemann domain over $M$. Since $\wdtl p\circ\phi\equiv p$, we get $\phi(X)\subset\wdtl X\setminus Z_p$. 
Consequently, $\phi:(X,p)\too(\wdtl X\setminus Z_p,\wdtl p)$ is an $\CF$-extension. Thus, there exists a morphism 
$\tau:(\wdtl X\setminus Z_p,\wdtl p)\too(\wdht X,\wdht p)$ such that $\tau\circ\phi\equiv\wdht\phi$. 
Then $(\sigma\circ\tau)\circ\phi\equiv\sigma\circ\wdht\phi\equiv\phi$, which by the identity principle gives $\sigma\circ\tau=\id$.
Moreover, $\wdtl p\circ\sigma\circ\wdht\phi\equiv\wdtl p \circ\phi\equiv p\equiv\wdht p\circ\wdht\phi$. Consequently, $\wdtl p\circ\sigma\equiv\wdht p$, which implies that 
$\sigma(\wdht X)\subset\wdtl X\setminus Z_p$.
Hence $\tau\circ\sigma=\id$ and, therefore, $\tau$ is an isomorphism.

(b) If $\phi:(X,\CF)\too(\wdtl X\setminus Z_p,\CF^{\wdtl\phi})$ is the $\CF_C$-envelope of holomorphy, then there exists a holomorphic mapping
$\sigma:\wdtl X\too\wdtl X\setminus Z_p$ such that $\sigma\circ\phi\equiv\phi$. Then $\sigma=\id$, which gives $Z_p=\varnothing$.

(c) Put $\wdht X:=\wdtl X\setminus Z_p$. We have $\CO(\wdtl X)|_{\wdht X}=\CO(\wdht X)$. 
In particular, the spaces $\CO(\wdtl X)$ and $\CO(\wdht X)$ endowed with the Fr\'echet topologies of locally uniform convergence are isomorphic.

We known that $(\wdht X,\wdtl p)$ is Stein (cf.~Theorem \ref{ThmThThmStein}).
In particular, $\wdht X$ is holomorphically convex, i.e.~for every compact $K\subset\wdht X$ its holomorphically convex hull
$\wdht K_{\CO(\wdht X)}:=\{x\in\wdht X: \forall_{g\in\CO(\wdht X)}: |g(x)|\leq\sup_K|g|\}$ is compact.

Suppose that $Z_p\neq\varnothing$ and let $a\in Z_p$. Let $U$ be a relatively compact open neighborhood of $a$. Then there exists a compact set $K\subset\wdht X$ 
such that $U\subset\wdht K_{\CO(\wdtl X)}$ (cf.~\cite{JarPfl2000}, Remark 1.4.5(l)). 
Thus $U\setminus Z_p\subset\wdht K_{\CO(\wdtl X)}\setminus Z_p=\wdht K_{\CO(\wdht X)}\subset\subset\wdht X$.
Consequently, $a\in U\subset\overline{U\setminus Z_p}\subset\wdht X$ --- a contradiction.
\end{proof}

\begin{Theorem}\label{ThmBihStein}
Let $(X,p)$, $(X,q)\in\FR_c(M)$ be such that $\str(X,p)=\str(X,q)$ and let $\CF\subset\CO(X)$ be such that $p, q\in\CF^N$. 
Let $\phi_p:(X,p)\too(\wdht X_p,\wdht p)$ and $\phi_q:(X,p)\too(\wdht X_q,\wdht q)$ be $\CF$-envelopes of holomorphy (over $M$).
Let $\phi:(X,\CF)\too(\wdtl X,\wdtl\CF)$ be the $\CF_C$-envelope of holomorphy and let $Z_p$ and $Z_q$ be as in Theorem \ref{ThmMain}.
Then the following conditions are equivalent:
\begin{myrenumerate}
\item there exists a biholomorphic mapping $\tau:\wdht X_p\too\wdht X_q$ such that $\tau\circ\phi_p\equiv\phi_q$;
\item $Z_p=Z_q$.
\end{myrenumerate}
\end{Theorem}

\begin{proof}
By Theorem \ref{ThmMain} we may assume that $\phi_p=\phi_q=\phi$, $(\wdht X_p,\wdht p)=(\wdtl X\setminus Z_p,\wdtl p)$, and $(\wdht X_q,\wdht q)=(\wdtl X\setminus Z_q,\wdtl q)$. 
If $Z_p=Z_q$, then we take $\tau:=\id$. Conversely, if $\tau:\wdtl X\setminus Z_p\too\wdtl X\setminus Z_q$ is such that $\tau\circ\phi\equiv\phi$, then $\tau=\id$ and hence $Z_p=Z_q$.
\end{proof}

\section{Proof of Example 1.3}\label{sectionExmpl}
In view of Theorem \ref{ThmMain}(a) the only problem is to prove that $\id:(\DD_\ast,\CH^\infty(\DD_\ast))\too(\DD,\CH^\infty(\DD))$ is the $\CH^\infty(\DD_\ast)_C$-envelope of holomorphy. It is clear that it is a 
$C$-morphism. Let $\phi:(\DD_\ast,\CH^\infty(\DD_\ast))\too(\wdtl X,\wdtl\CF)$ be the $\CH^\infty(\DD_\ast)_C$-envelope of holomorphy. Then there exists a $C$-morphism 
$\sigma:(\DD,\CH^\infty(\DD))\too(\wdtl X,\wdtl\CF)$ such that $\sigma=\phi$ on $\DD_\ast$. Let $\wdtl p\in\wdtl\CF$ be such that $\wdtl p\circ\sigma\equiv\id$. 
Observe that $\wdtl p(\wdtl X)\subset\DD$. In fact, since $\wdtl p\not\equiv\const$, we only need to show that $\wdtl p(\wdtl X)\subset\overline\DD$.
Suppose that $z_0\in\wdtl p(\wdtl X)\setminus\overline\DD$. Put $f(z):=1/(z-z_0)$. Then $f\in\CH^\infty(\DD)$. Let $\wdtl f\in\wdtl\CF$ be such that $\wdtl f\circ\sigma\equiv f$. 
Thus, by the identity principle, $\wdtl f\cdot(\wdtl p-z_0)\equiv1$ --- a contradiction. 

Finally, $\sigma:\DD\too\wdtl X$ is biholomorphic and $\sigma^{-1}=\wdtl p$.  

\section{Proof of Example 1.3}\label{sectionExmpl-2}
\footnote{To see that $\wdtl X$ is connected it suffices to observe that $\CC\setminus\{0\}\ni\lambda\tuu(\frac12(\lambda+1/\lambda), \frac{i}2(\lambda-1/\lambda))\in\wdtl X$ 
is a global parametrization.}
It is clear that $\phi:(X,\CF)\too(\wdtl X,\wdtl\CF)$ is a $C$-morphism. Suppose that $\phi^0:(X,\CF)\too(X^0,\CF^0)$ is the $\CF_C$-envelope of holomorphy and let $\sigma:(\wdtl X,\wdtl\CF)\too(X^0,\CF^0)$
be a $C$-morphism such that $\sigma\circ\phi\equiv\phi^0$. Let $f^0_j\in\CF^0$ be such that $f^0_j\circ\phi^0\equiv f_j$, $j=1,2$. Then $((f^0_1)^2+(f^0_2)^2)\circ\phi^0\equiv1$, which shows that
$f^0:=(f^0_1,f^0_2):X^0\too\wdtl X$. Moreover, $\sigma\circ f^0\circ\phi^0=\sigma\circ\phi=\phi^0$. Consequently, $\sigma\circ f^0\equiv\id$, which implies that $\sigma$ is a $C$-isomorphism.

\section{Proof of Theorem 1.4}\label{sectionThmIsNotStein}
Let  $Y:=\{(z_1,\dots,z_{n+1})\in\CC^{n+1}: z_1^2+\dots+z_{n+1}^2=0\}$, $Y_0:=Y\setminus\{0\}$. Observe that $Y_0$ is a connected
\footnote{To see that $Y_0$ is connected we may argue as follows. Since $Y_0$ is a $\CC_\ast$-cone ($\CC_\ast:=\CC\setminus\{0\}$), it suffices to show that the any two points from 
$Q:=Y_0\cap\partial\BB_{n+1}$ may be joined
in $Y_0$ with a continuous curve, where $\BB_k:=\{z\in\CC^k: \|z\|<1\}$ is the unit Euclidean ball. We have $Q=\{x+iy\in\RR^{n+1}+i\RR^{n+1}: \|x\|=\|y\|=1/\sqrt{2},\;\sprod{x}{y}=0\}$. 
Any orthogonal operator $A\in\OO(n+1,\RR)$ acts on $Q$ according to the formula
$x+iy\tuu Ax+iAy$. Since the special orthogonal group $\SS\LL(n+1,\RR)$ is connected, each point $a\in Q$ may be joined in $Q$ with a point $b$ of the form $(b',0)+i(0',b_n)$ with 
$\|b'\|=|b_n|=1/\sqrt{2}$. Taking a suitable rotation $\zeta\in\partial\DD$ we may join $b$ in $Y_0$ with a point $c=(c',c_n):=\zeta b$ such that $\|c'\|=c_n=1/\sqrt{2}$. It remains to use the fact that
$\partial\BB_n(1/\sqrt{2})$ is connected.} 
$n$-dimensional complex manifold 
Let $X_0:=Y_0\setminus M_{n+1}$, where $M_{n+1}:=\{(z_1,\dots,z_{n+1})\in Y_0: z_{n+1}=0\}$, $p_0(z_1,\dots,z_{n+1}):=(z_1,\dots,z_n)$. Obviously, $M_{n+1}$ is a one-codimensional analytic subset of $Y_0$. 
Then $(X_0,p_0)\in\FR_c(\CC^n)$. Put
$\CF_0:=\CH^\infty(X_0)\cup\{z_j|_{X_0}: j=1,\dots,n+1\}$. 
\halfskip

(a) First we will prove that $\id:(X_0,\CF_0)\too(Y_0,\CG)$ is the $(\CF_0)_C$-envelope of holomorphy, where $\CG:=\CH^\infty(Y_0)\cup\{z_j|_{Y_0}: j=1,\dots,n+1\}$. 

By the Riemann removable singularities theorem we see that $\id:(X_0,\CF_0)\too(Y_0,\CG)$ is the $(\CF_0)_C$-extension. 
Let $\phi_0:(X_0,\CF_0)\too(\wdtl X_0,\wdtl\CF_0)$ be the $(\CF_0)_C$-envelope of holomorphy. Observe that 
$\wdtl\CF_0=\CH^\infty(\wdtl X_0)\cup\{F_1,\dots,F_{n+1}\}$, where $F_j\in\CO(\wdtl X_0)$ is such that $F_j\circ\phi_0\equiv z_j|_{X_0}$, $j=1,\dots,n+1$. 
Let $\sigma:(Y_0,\CG)\too(\wdtl X,\wdtl\CF)$ be a $C$-morphism with $\sigma=\phi_0$ on $X_0$. Clearly, $F_j\circ\sigma\equiv z_j|_{Y_0}$, $j=1,\dots,n+1$. 
Let $N:=\{x\in\wdtl X_0: F_j(x)=0,\;j=1,\dots,n+1\}$; $N$ is an analytic subset of $\wdtl X_0$ with $\dim N\leq n-1$. 
Put $F:=(F_1,\dots,F_{n+1}):\wdtl X_0\too\CC^{n+1}$. Observe that $F:\wdtl X_0\too Y$. It is clear that $\sigma:Y_0\too\wdtl X_0\setminus N$ is
biholomorphic ($\sigma^{-1}=F$). Thus, it remains to show that $N=\varnothing$. Suppose that $a\in N$. By the definition of the class $\ES$ there exist an open neighborhood $U\subset\wdtl X_0$ and
a mapping $\wdtl f^0=(\wdtl f^0_1,\dots,\wdtl f^0_n)\in\wdtl\CF^n$ such that $\wdtl f^0|_U:U\too\wdtl f^0(U)$ is biholomorphic. There are essentially the following three possibilities:

\bu $\wdtl f^0_j\in\CH^\infty(\wdtl X)$, $j=1,\dots,n$.

Since $Y_0$ is a $\CC_\ast$-cone, we easily conclude that for every $f\in\CH^\infty(Y_0)$ 
we get $f(\lambda b)=f(b)$, $\lambda\in\CC_\ast$, $b\in Y_0$. Consequently, if $\wdtl f\in\CH^\infty(\wdtl X_0)$, then $\wdtl f(\sigma(\lambda b))=\wdtl f(\sigma(b))$, $\lambda\in\CC_\ast$, $b\in Y_0$. 
In particular, $\wdtl f^0(\sigma(\lambda b))=\wdtl f^0(\sigma(b))$, $\lambda\in\CC_\ast$, $b\in Y_0$. Taking $b\in\sigma^{-1}(U\setminus N)$ we get a contradiction with injectivity of $\wdtl f^0|_U$.

\bu $\wdtl f^0_1=F_1,\dots,\wdtl f^0_s=F_s$, $\wdtl f^0_{s+1},\dots,\wdtl f^0_n\in\CH^\infty(\wdtl X_0)$ for some $1\leq s\leq n-1$.

Using the above argument we see that if $b\in Y_0$ is such that $b_1=\dots=b_s=0$, then $\wdtl f^0(\sigma(\lambda b))=\wdtl f^0(\sigma(b))$, $\lambda\in\CC_\ast$. Thus, it is enough to show that
$(U\setminus N)\cap N_1\cap\dots\cap N_s\neq\varnothing$, where $N_j:=\{x\in\wdtl X_0: F_j(x)=0\}$, $j=1,\dots,n+1$. Suppose the contrary. 
Then $N_1\cap\dots\cap N_s\cap U\subset N_{s+1}\cap\dots\cap N_{n+1}$. Let $W:=\sigma^{-1}(U\setminus N)\neq\varnothing$. We have $\{z\in W: z_1=\dots=z_s=0\}\subset\{z\in W: z_{s+1}=\dots=z_{n+1}=0\}$.
Since $n+1-s\geq2$, we get a contradiction.

\bu  $\wdtl f^0_j=F_j$, $j=1,\dots,n$.

Using local complex coordinates $(\zeta_1,\dots,\zeta_n)$ in a neighborhood of a we have $\det[\pder{F_j}{\zeta_k}]_{j,k=1,\dots,n}\neq0$. Since $F_1^2+\dots+F_{n+1}^2\equiv0$, we have
$F_1\pder{F_1}{\zeta_k}+\dots+F_n\pder{F_n}{\zeta_k}=-F_{n+1}\pder{F_{n+1}}{\zeta_k}$, $k=1,\dots,n$. Thus, by Cramer's formulas, $F_j=\Phi_jF_{n+1}$ in a neighborhood of $a$,  
where $\Phi_j$ is holomorphic, $j=1,\dots,n$. Consequently, $\pder{F_j}{\zeta_k}(a)=\Phi_j(a)\pder{F_{n+1}}{\zeta_k}(a)$, $j,k=1,\dots,n$, which gives a contradiction.

\halfskip

(b) Now we will prove that for every circular compact $K\subset Y_0$ we have $\DD_\ast\cdot K\subset\wdht K_{\CO(Y_0)}$, which directly implies that $Y_0$  is not holomorphically convex.
It suffices to prove that for any $f\in\CO(Y_0)$ and $b\in Y_0$ the function $\CC_\ast\ni\lambda\overset{f_b}\tuu f(\lambda b)$ is bounded near $\lambda=0$ 
(and consequently extends holomorphically to $\CC$). Indeed, then for every circular compact set $K\subset Y_0$, $f\in\CO(Y_0)$, $b\in K$ and $\lambda\in\DD_\ast$ we have
$$
|f(\lambda b)|=|f_b(\lambda)|\leq\max_{\partial\DD}|f_b|=\max_{\lambda\in\partial\DD}|f(\lambda b)|\leq\max_K|f|.
$$

Fix an $f$. 
Let $M:=\{(w_1,\dots,w_n)\in\CC^n: w_1^2+\dots+w_n^2=0\}$; $M$ is a one-codimensional analytic subset of $\CC^n$. Define 
\begin{multline*}
\CC\times(\CC^n\setminus M)\ni(\xi,w)\tuu\Big(\xi-f\big(w,\sqrt{-(w_1^2+\dots+w_n^2)}\big)\Big)\Big(\xi-f\big(w,-\sqrt{-(w_1^2+\dots+w_n^2)}\big)\Big)\\
=:\xi^2+B(w)\xi+C(w),
\end{multline*}
where $B, C\in\CO(\CC^n\setminus M)$. Observe that $B, C$ are locally bounded in $\CC^n\setminus\{0\}$ --- if $w^0\in M\setminus\{0\}$ and $\CC^n\setminus M\ni w^s\too w^0$, then $B(w^s)\too-2f(w^0,0)$ and
$C(w^s)\too f^2(w^0,0)$. Thus we may first assume that $B,C\in\CO(\CC^n\setminus\{0\})$ and next, by the Hartogs theorem, that $B,C\in\CO(\CC^n)$. In particular, the function $\Delta:=B^2-4C$ is holomorphic 
on $\CC^n$. This implies that the roots $\xi_\pm(w)=\frac12(-B(w)\pm\sqrt{\Delta(w)})$ are locally bounded with respect to $w\in\CC^n$. Finally,
$f(\lambda b)\in\{\xi_-(\lambda b_1,\dots,\lambda b_n), \xi_+(\lambda b_1,\dots,\lambda b_n)\}$ is bounded near $\lambda=0$.

\halfskip

(c) Using Theorem \ref{ThmMain}(a), we conclude that $\wdtl p_0=(z_1|_{Y_0},\dots,z_n|_{Y_0})$. Hence $Z_{p_0}=M_{n+1}$ and, consequently, 
$(X_0,p_0)$ is an $\CF_0$-domain of holomorphy. In particular, $X_0$ is Stein.

\halfskip

(d) Observe that $(X_0,p_0)$ is a two-fold cover over $\CC^n$, i.e.~every point $a\in p_0(X_0)$ has an open neighborhood $U$ such that $p_0^{-1}(U)=U_1\cap U_2$, $U_1\cap U_2=\varnothing$,
and $p_0|_{U_j}:U_j\too U$ is biholomorphic, $j=1,2$. Hence, there exists a domain $X\subset\CC^n$ such that $\phi:(X,\id_X)\too(X_0,p_0)$ is the 
$\CO(X)$-envelope of holomorphy in the sense of Riemann domains (cf.~\cite{ForZam1983}, see also \cite{JarPfl2000}, Theorem 4.5.18). Put $\CF:=\phi^\ast(\CF_0)$. It remains to prove that $\phi:(X,\CF)\too(Y_0,\CG)$ is the $\CF_C$-envelope of holomorphy. It is obviously a $C$-morphism.
Let $\wdtl\phi:(X,\CF)\too(\wdtl X,\wdtl\CF)$ be the $\CF_C$-envelope of holomorphy. Then there exists a $C$-morphism $\sigma:(Y_0,\CG)\too(\wdtl X,\wdtl\CF)$ such that 
$\sigma\circ\phi\equiv\wdtl\phi$. Since $\sigma|_{X_0}:(X_0,\CF_0)\too(\wdtl X,\wdtl\CF)$ is a $C$-morphism, using (a) we conclude that there exists a $C$-morphism $\tau:(\wdtl X,\wdtl\CF)\too(Y_0,\CG)$ such that 
$\tau\circ\sigma|_{X_0}\equiv\id_{X_0}$. Thus $\tau\circ\sigma\equiv\id$ and hence $\tau$ is biholomorphic ($\tau^{-1}=\sigma$).

\halfskip

(e) We are going to prove that $\id:(Y_0,\CO(Y_0))\too(Y_0,\CO(Y_0))$ is the $\CO(Y_0)_C$-envelope of holomorphy. Suppose that $\alpha:(Y_0,\CO(Y_0))\too(\wdtl Y_0,\CO(\wdtl Y_0))$ is the
$\CO(Y_0)_C$-envelope of holomorphy. Then, by (a),  $\alpha:(X_0,\CF_0)\too(\wdtl Y_0,\CF_0^\phi)$ is a $C$-morphism. Consequently,  there exists a $C$-morphism 
$\sigma:(\wdtl Y_0,\CF_0^\phi)\too(Y_0,\CG)$ such that $\sigma\circ\alpha\equiv\id$.

\halfskip

(f) Suppose that $Y_0$ is a Riemann domain over $\CC^n$, i.e.~there exists a locally biholomorphic mapping $q:Y_0\too\CC^n$. By Theorem \ref{ThmThThmStein}, to get a contradiction it suffices 
to prove that $(Y_0,q)$ is a domain of holomorphy. Suppose that $\alpha:(Y_0,q)\too(Z,r)$ is an $\CO(Y_0)$-extension. 
Then $\alpha:(Y_0,\CO(Y_0))\too(Z,\CO(Z))$ is a $C$-morphism. Consequently, by (e), there exists a $C$-morphism $\sigma:(Z,\CO(Z))\too(Y_0,\CO(Y_0))$ such that $\sigma\circ\alpha\equiv\id$. In particular,
$q\circ\sigma\circ\alpha\equiv q\equiv r\circ\alpha$. Thus $\sigma:(Z,r)\too(Y_0,q)$ is a morphism.
\qed

\bibliographystyle{amsplain}

\end{document}